\documentclass[12pt,reqno]{amsart}

\usepackage{amsfonts,amsmath,amsthm}
\usepackage{amssymb,epsfig, verbatim}
\usepackage{ascmac} 
\usepackage{cases}
\usepackage[usenames]{color}
\usepackage{enumerate} 

\usepackage{color} 
\definecolor{vert}{rgb}{0,0.6,0}

\theoremstyle{plain}
\newtheorem{thm}{Theorem}[section]

\newtheorem{defn}{Definition}

\newtheorem{lem}[thm]{Lemma}
\newtheorem{cor}[thm]{Corollary}
\newtheorem{prop}[thm]{Proposition}
\theoremstyle{remark}
\newtheorem{rem}{\bf{Remark}}
\numberwithin{equation}{section}

\newcommand{\N}{\mathbb{N}}

\newcommand{\R}{\mathbb{R}}
\newcommand{\bS}{\mathbb{S}}

\newcommand{\vep}{\varepsilon}

\newcommand{\al}{\alpha}

\newcommand{\sig}{\sigma}

\newcommand{\Gam}{\Gamma}

\newcommand{\ol}{\overline}
\newcommand{\ul}{\underline}
\newcommand{\pl}{\partial}

\begin{document}

\title[A scheme for time fractional equations]{\protect{On a discrete scheme for time fractional \\
fully nonlinear evolution equations}}
\author[Y. GIGA, Q. LIU, H. MITAKE]
{Yoshikazu Giga, Qing Liu, Hiroyoshi Mitake}

\thanks{
The work of YG was partially supported by Japan Society for the Promotion of Science (JSPS) through grants KAKENHI \#26220702,  \#16H03948, \#18H05323, \#17H01091.
The work of QL was partially supported by the JSPS grant KAKENHI \#16K17635 and the grant  \#177102 from Central Research Institute of Fukuoka University. 
The work of HM was partially supported by the JSPS grant  KAKENHI \#16H03948.
}

\address[Y. Giga]{
Graduate School of Mathematical Sciences, 
University of Tokyo 
3-8-1 Komaba, Meguro-ku, Tokyo, 153-8914, Japan}
\email{labgiga@ms.u-tokyo.ac.jp}

\address[Q. Liu]{Department of Applied Mathematics, Faculty of Science, Fukuoka University, Fukuoka 814-0180, Japan. 
}
\email{qingliu@fukuoka-u.ac.jp}

\address[H. Mitake]{
Graduate School of Mathematical Sciences, 
University of Tokyo 
3-8-1 Komaba, Meguro-ku, Tokyo, 153-8914, Japan}
\email{mitake@ms.u-tokyo.ac.jp}



\maketitle

\begin{abstract}
We introduce a discrete scheme for second order fully nonlinear parabolic PDEs with Caputo's time fractional derivatives. We prove the convergence of the scheme in the framework of the theory of viscosity solutions. 
The discrete scheme can be viewed as a resolvent-type approximation. 

\end{abstract}




\section{Introduction} 
In this paper, we are concerned with the second order fully nonlinear PDEs with Caputo's time fractional derivatives: 
\begin{numcases}{}
\pl_t^{\al}u(x,t)+F(x, t, Du, D^2u)=0 &\qquad\text{for all $x\in\R^n, t>0,$} \label{eq:1}\\
u(x,0)=u_0(x) &\qquad\text{for all $x\in\R^n$},  \label{eq:ini}
\end{numcases}
where $\alpha\in(0,1)$ is a given constant,  $u:\R^n\times[0,\infty)\to\R$ is an unknown function and 
$Du$ and $D^2 u$, respectively, denote its spatial gradient and Hessian of $u$. 
We \textit{always} assume that $u_0\in BUC(\R^n)$, 
which denotes the space of all bounded uniformly continuous functions in $\R^n$. We denote  
\textit{Caputo's time fractional derivative} by $\pl_t^{\al}u$, i.e., 
\[
\pl_t^{\al}u(x,t):=\frac{1}{\Gam(1-\al)}\int_0^t(t-s)^{-\al}\pl_{s}u(x,s)\,ds, 
\]
where $\Gam$ is the Gamma function.

We assume that $F$ is a continuous \textit{degenerate elliptic} operator, that is,
\[
F(x, t, p, X_1) \leq F(x, t, p, X_2)  
\]
for all $x\in\R^n, t\geq 0, p \in \R^n$ and $X_1, X_2 \in \bS^n \text{ with } X_1 \geq X_2$, where $\bS^n$ denotes the space of $n \times n$ real symmetric matrices. 
Moreover, throughout this work we assume that $F$ is locally bounded in the sense that 
\begin{equation}\label{eq:op bound}
M_R:=\sup_{\substack{(x, t)\in \R^n\times [0, \infty)\\ |p|, |X|\leq R}}|F(x, t, p, X)|<\infty\qquad \text{for any $R>0$}.
\end{equation}

Studying differential equations with fractional derivatives is motivated by mathematical models that 
describe diffusion phenomena in complex media like fractals, which is sometimes called \textit{anomalous diffusion} 
(see \cite{MK} for instance). 
It has inspired further research on numerous related topics. 
We refer to a non-exhaustive list of references \cite{L, SY, ACV, C, GN, TY, A, N, KY, CKKW} and the references therein.  

Among these results, the authors of \cite{ACV, A} mainly study regularity of solutions to a space-time nonlocal equation with 
Caputo's time fractional derivative in the framework of viscosity solutions. 
More recently, unique existence of a viscosity solution to the initial value problem with Caputo's time fractional derivatives has been established in the thesis of Namba \cite{N-thesis} and independently and concurrently by Topp and Yangari \cite{TY}. The main part of \cite{N-thesis} on this subject has been published in \cite{GN, N}. For example, a comparison principle, Perron's method, and stability results for \eqref{eq:1} in bounded domains with various boundary conditions have been established in \cite{GN, N}. Similar results for whole space has been established in \cite{TY} for nonlocal parabolic equations.

Motivated by these works, in this paper we introduce a discrete scheme for \eqref{eq:1}--\eqref{eq:ini}, which will be explained in detail in the subsection below.

\subsection{The discrete scheme}\label{subsec:scheme} 
Our scheme is naturally derived from the definitions of Riemann integral and Caputo's time fractional derivative. 
We first observe that 
\begin{align*}
\pl_t^{\al}u(\cdot, mh)&
=\frac{1}{\Gam(1-\al)}\int_0^{mh}(mh-s)^{-\al}\pl_{s}u(x,s)\,ds\\
&=\frac{1}{\Gam(1-\al)}\sum_{k=0}^{m-1} \int_{kh}^{(k+1)h}(mh-s)^{-\al}\pl_{s}u(x,s)\,ds  
\end{align*}
for $m\in\N$ and $h>0$. 
If $u$ is smooth in $\R^n\times [0, \infty)$ and $h$ is small, then we can approximately think that 
\[
\int_{kh}^{(k+1)h}(mh-s)^{-\al}\pl_{s}u(x,s)\,ds
\approx
\int_{kh}^{(k+1)h}(mh-s)^{-\al}\frac{u(x,(k+1)h)-u(x,kh)}{h}\,ds. 
\]
Note that  $z\Gam(z)=\Gam(z+1)$ and
\begin{align*}
\int_{kh}^{(k+1)h}(mh-s)^{-\al}\,ds
=&\, 
\frac{1}{1-\al}\left(((m-k)h)^{1-\al}-((m-k-1)h)^{1-\al}\right)\\  
=&\,
\frac{1}{1-\al}f(m-k)h^{1-\al}, 
\end{align*}
where we set 
\begin{equation}\label{func:f}
f(r):=r^{1-\al}-(r-1)^{1-\al}\quad\text{for} \ r\ge1.  
\end{equation}
Thus, 
\begin{align*}
\pl_t^{\al}u(\cdot, mh)\approx 
&\, 
\frac{1}{\Gam(2-\al)h^\al}\sum_{k=0}^{m-1}
f(m-k)\left(u(x,(k+1)h)-u(x,kh)\right)\\
&\,
= 
\frac{1}{\Gam(2-\al)h^\al}\left\{
u(x,mh)
-\sum_{k=0}^{m-1} C_{m, k}u(x,kh)\right\}, 
\end{align*}
where we set 
\[
C_{m,0}:=f(m), \quad C_{m,k}:=f(m-k)-f(m-(k-1)) \quad\text{for} \ k=1,\ldots, m-1.
\]
Since $f$ is a non-increasing function, we easily see that  
\begin{equation}\label{positive}
C_{m,k}\ge 0 \quad\text{for} \ k=0,\ldots, m-1,   
\end{equation}
which implies monotonicity of the scheme 
(see Proposition \ref{prop:monotone}).

Inspired by this observation, for any fixed $h>0$, 
we below define a family of functions 
$\{U^h(\cdot, mh)\}_{m\in\N\cup\{0\}}\subset BUC(\R^n)$ by induction. 
Set $U^h(\cdot, 0):=u_0^h$, 
where $u_0^h\in BUC(\R^n)$ satisfies 
\begin{equation}\label{initial approx}
\sup_{\R^n} \left|u_0^h-u_0\right|\to 0\quad \text{as $h\to 0$.}
\end{equation}
Let $U^h(\cdot, mh)\in C(\R^n)$ for $m\geq 1$ be the viscosity solution of  
\begin{equation}\label{eq:m}
\frac{1}{\Gam(2-\al)h^\al}\left\{
u(x)
-\sum_{k=0}^{m-1} C_{m,k}U^h(x,kh)\right\}
+F\left(x, t, D u, D^2u\right)=0 \quad \text{in $\R^n$}.
\end{equation}  
Let us emphasize here that the equation \eqref{eq:m} is an (degenerate) elliptic problem with the elliptic operator strictly monotone in $u$. 
In fact, for any $m\geq 1$  the elliptic equation is of the form
\begin{equation}\label{eq:elliptic}
\lambda u(x)+F\left(x, t, Du, D^2u\right)=g(x)\quad \text{in $\R^n$,}
\end{equation}
where $\lambda>0$ and $g\in BUC(\R^n)$.  We can obtain such a unique  viscosity solution $U^h(\cdot, mh)\in BUC(\R^n)$ to \eqref{eq:m} with $t=mh$ for any $m\in \N$ under appropriate assumptions on $F$. 

Define the function $u^{h}:\R^n\times[0,\infty)\to\R$ by 
\begin{equation}\label{def:Uh}
u^h(x,t):=U^h(x, mh)\quad\text{for each}\ x\in \R^n, \ t\in[mh,(m+1)h), \ m\in\N\cup\{0\}.  
\end{equation}
Our main result of this paper is to show the convergence of  $u^h$ to the unique viscosity solution of \eqref{eq:1}--\eqref{eq:ini}.

We remark that our scheme can be regarded as a resolvent-type approximation.  Recall the implicit Euler scheme for the differential equation: 
\[
u_t + F[u]:=u_t+F\left(x, t, Du,D^2u\right) = 0, 
\]
which is given by 
\[
u^h(\cdot, mh) - u^h(\cdot, (m-1)h) + hF[u^h(\cdot, mh)]=0\quad (m\in \N). 
\] 
This is a typical scheme by approximating $u$ by a function $u^h$ piecewise linear in time with time grid length $h$. The resulting equation is a resolvent type 
equation for $u^h(\cdot, mh)$ if $u^h(\cdot, (m-1)h)$ is given. It is elliptic if the 
original equation is parabolic. 

\subsection{Main Results}

We first give an abstract framework on the convergence of $u^h$. 

\begin{thm}[Scheme convergence]\label{thm:main}
Assume that \eqref{eq:op bound} and the following two conditions hold.
\begin{enumerate}
\item[{\rm(H1)}] For any $g\in BUC(\R^n)$, there exists a viscosity solution $u\in BUC(\R^n)$ to \eqref{eq:elliptic} for any $t>0$. Moreover, if $u, v\in BUC(\R^n)$ are, respectively, a subsolution and a supersolution of \eqref{eq:elliptic} with any fixed $t>0$, then $u\leq v$ in $\R^n$. 
\item[{\rm(H2)}] Let $u\in USC(\R^n\times [0, \infty))$ and $v\in LSC(\R^n\times [0, \infty))$ be, respectively, a sub- and a supersolution of \eqref{eq:1}. Assume $u$ and $v$ are bounded in $\R^n\times [0, T)$ for any $T>0$. 
If $u(\cdot, 0)\leq v(\cdot, 0)$ in $\R^n$, 
then $u\leq v$ in $\R^n\times [0, \infty)$.
\end{enumerate}
Let $u^h$ be given by \eqref{def:Uh} for any $h>0$, where initial data $u^h_0$ is assumed to fulfill \eqref{initial approx}.
Then, $u^h\to u$ locally uniformly in $\R^n\times [0, \infty)$ as $h\to 0$, where $u$ is the unique viscosity solution to \eqref{eq:1}--\eqref{eq:ini}. 
\end{thm}
We obtain the following corollary of Theorem \ref{thm:main} under more explicit sufficient conditions of (H1) and (H2). 

\begin{cor}\label{cor:non-periodic}
Assume that \eqref{eq:op bound} and the following two conditions hold. 
\begin{enumerate}
\item[{\rm(F1)}] There exists a modulus of continuity $\omega: [0, \infty)\to [0, \infty)$ such that
\[
F\left(x, t, \mu(x-y),Y\right)-F\left(y, t, \mu(x-y),X\right)\le \omega\left(|x-y|(\mu|x-y|+1)\right)
\]
for all $\mu>0$, $x, p\in\R^n$, $t\geq 0$ and $X,Y\in\bS^n$  satisfying 
\[
\left(
\begin{array}{cc}
X & 0 \\
0 & -Y
\end{array}
\right)
\le 
\mu
\left(
\begin{array}{cc}
I & -I \\
-I & I
\end{array}
\right). 
\]
\item[{\rm(F2)}] There exists a modulus of continuity $\tilde{\omega}: [0, \infty)\to [0, \infty)$ such that 
\[
|F(x, t, p, X)-F(x, t, q, Y)|\leq \tilde{\omega}(|p-q|+|X-Y|)
\] 
for all $x\in \R^n$, $t\geq 0$, $p, q\in \R^n$ and $X, Y\in \bS^n$. 
\end{enumerate} 
Then, the conclusion of Theorem \ref{thm:main} holds. 
\end{cor}
\begin{rem}
The assumption (F2) can be removed in the presence of periodic boundary condition, that is, $x\mapsto u_0(x)$ and $x\mapsto F(x, t, p, X)$ 
are periodic with the same period. 
Recall that in a bounded domain or with the periodic boundary condition,  (H1) is established in \cite{CIL} 
and (H2) 
 is available in \cite[Theorem 3.1]{GN}  \cite[Theorem 3.4]{N} under (F1).


\end{rem}

The comparison result in (H1) under (F1), (F2) in an unbounded domain is due to \cite{JLS}.  Existence of solutions in this case can be obtained by Perron's method. In fact, thanks to \eqref{eq:op bound} with $R=0$, we can take $C>0$ large such that $C$ and $-C$ are, respectively, a supersolution and a subsolution of \eqref{eq:elliptic}. We then can prove the existence of solutions by adopting the standard argument in \cite{CIL, G}. 
In addition, as shown in \cite{TY}, (H2) is also guaranteed by (F1) and (F2). 

Our results above apply to a general class of nonlinear parabolic equations. We refer the reader to \cite[Example 3.6]{CIL} for concrete examples of $F$ that satisfy our assumptions, especially the condition (F1). 


Finally, it is worthwhile to mention that the idea for a discrete scheme in this paper can be adopted to handle a more general type of time fractional derivatives as in \cite{C, CKKW},  provided that the comparison theorems can be obtained. 
In this paper, we choose Caputo's time fractional derivatives to simplify the presentation.

\smallskip
This paper is organized as follows. In Section \ref{sec:pre}, we give the monotonicity and boundedness of discrete schemes. Section \ref{sec:main} is devoted to the proof of Theorem \ref{thm:main}.

\section{Preparations}\label{sec:pre}
We first recall the definition of viscosity solutions to \eqref{eq:1}.
\begin{defn}[Definition of viscosity solutions]\label{defn vis}
For any $T>0$, a function $u\in USC(\R^n\times[0,T))$ {\rm(}resp., $u\in LSC(\R^n\times[0,T))${\rm)}
is called a viscosity subsolution {\rm(}resp., supersolution{\rm)} of \eqref{eq:1} 
if for any $\phi\in C^{2}(\R^n\times [0,T))$ 
one has 
\begin{equation}\label{def:vis}
\partial_t^\alpha \phi(x_0, t_0)+F(x_0, t_0, D \phi(x_0, t_0), D^2 \phi(x_0, t_0)) 
\le {\rm(resp.,}\ge{\rm)}\ 0
\end{equation}
whenever $u-\phi$ attains a local maximum {\rm(}resp., minimum{\rm)} at $(x_0, t_0)\in \R^n\times (0,T)$. 

We call $u\in C(\R^n\times[0,T))$ a viscosity solution of \eqref{eq:1} if 
$u$ is both a viscosity subsolution and a supersolution of \eqref{eq:1}. 

\end{defn}

\begin{rem}
Our definition essentially follows \cite[Definition 2.2]{N}. In fact, since 
\begin{equation}\label{equiv time der}
\pl_t^{\al}\phi(x, t)
=
\frac{1}{\Gam(1-\al)}\left(\frac{\phi(x, t)-\phi(x, 0)}{t^\al}+\al\int_0^t\frac{\phi(x, t)-\phi(x, s)}{(t-s)^{1+\al}}\,ds \right) 
\end{equation}
for any $\phi\in C^{1}(\R^n\times [0, \infty))$, our definition is thus the same as \cite[Definition 2.2]{N}. A similar definition of viscosity solutions is cocurrently and independently proposed in \cite[Definition 2.1]{TY} for general space-time nonlocal parabolic problems. 

Another possible way to define sub- or supersolutions is to separate the term $\pl_t^{\al}\phi(x, t)$ in \eqref{def:vis} into two parts like \eqref{equiv time der} and replace $\phi$ in one or both of the parts by $u$. 
See \cite[Definition 2.1]{TY} and \cite[Definition 2.5]{GN}. 
Such definitions are proved to be equivalent to Definition \ref{defn vis}. We refer to \cite[Lemma 2.3]{TY} and \cite[Proposition 2.5]{N} for proofs. 
Note that the original definition of viscosity solutions in \cite{N-thesis, GN} looks stronger but it turns out that it is the same \cite[Lemma 2.9, Proposition 3.6]{N-thesis}.
\end{rem}

For any $h>0$, define $\pl_{t}^{\al,h}: L^\infty_
{loc}(\R^n\times [0, \infty))\to L^\infty_{loc}(\R^n\times [0, \infty))$ to be 
\begin{equation}\label{eq:op discrete}
\pl_{t}^{\al,h}u(x, t):=\frac{1}{\Gam(2-\al)h^\al}\left\{
u(x, mh)
-\sum_{k=0}^{m-1} C_{m,k}u(x,kh)\right\}
\end{equation}
for $(x, t)\in \R^n\times [0, \infty)$, and $m\in \N\cup\{0\}$ satisfying $m=\lfloor t/h\rfloor$,  
where $\lfloor s\rfloor$ denotes the greatest integer less than or equal to $s\geq 0$. 

A locally bounded function $u: \R^n\times [0, \infty)\to \R$ is said to be a subsolution (resp., supersolution) of 
\begin{equation}\label{eq:discrete}
\pl_t^{\al, h} u+F\left(x, t, Du, D^2u\right)=0 \quad \text{in $\Omega\times (0, \infty)$}
\end{equation}
if for any $m\in \N$, $U=u(\cdot, mh)\in USC(\R^n)$ (resp., $U=u(\cdot, mh)\in LSC(\R^n)$) is a viscosity subsolution (resp, supersolution) of
\[
\frac{1}{\Gam(2-\al)h^\al}\left\{
U-\sum_{k=0}^{m-1} C_{m,k}u(\cdot, kh)\right\}
+F\left(x, mh, D U, D^2U\right)=0 \quad \text{in $\R^n$.}
\]
By definition, it is clear that $u^h$ given by \eqref{def:Uh} is a solution of \eqref{eq:discrete}.

\begin{prop}[Monotonicity]\label{prop:monotone}
Fix $h>0$.  Assume that {\rm(H1)} holds.  Let $U^h(\cdot, t)$, $V^h(\cdot, t)\in BUC(\R^n)$ for all $t\geq 0$ be, respectively, a subsolution and supersolution to \eqref{eq:discrete}. 
Then, $U^h(\cdot, mh)\leq V^h(\cdot, mh)$ in $\R^n$ for all $m\in \N$. 
\end{prop}
\begin{proof}
Due to the positiveness \eqref{positive} of $C_{m,k}$, one can easily see that the scheme is monotone by iterating the comparison principle in (H1) for elliptic problems. 
\end{proof}
 
We next discuss below the boundedness of the scheme. 
\begin{lem}[Barrier]\label{lem:est1}
For any $h>0$, let $V^h(x, t):=(mh)^\alpha$ for all $x\in \R^n$ and $t\geq 0$ with $m=\lfloor t/h\rfloor$. 
Then, 
\[
\partial_t^{\alpha, h} V^h(x, t)\geq {(1-\alpha)\alpha\over \Gamma(2-\alpha)}\qquad \text{for all $x\in \R^n$ and $t\geq h$}. 
\]
\end{lem}
\begin{proof}
We have 
\begin{equation}\label{eq:barrier1}
\partial_t^{\alpha,h} V^h(x, t)=\frac{1}{\Gam(2-\alpha)}\sum_{k=0}^{m-1}f(m-k)\big((k+1)^\alpha-k^\alpha\big)
\end{equation}
for all $x\in \R^n$ and $t\geq h$.  Noting that 
\begin{align*}
& f(m-k)\ge (1-\alpha)/(m-k)^\alpha\ge (1-\alpha)/m^\alpha, \\
& (k+1)^\alpha-k^\alpha\ge \alpha/(k+1)^{1-\alpha}\ge\alpha/m^{1-\alpha},  
\end{align*}
we can plug these estimates into \eqref{eq:barrier1} to deduce the .  
\end{proof}

\begin{lem}[Uniform boundedness]\label{lem:bound} 
Assume that \eqref{eq:op bound} and {\rm(H1)} hold.  
Let $u^h$ be given by \eqref{def:Uh} for any fixed $h>0$. Then, 
\[
|u^h(x, t)|\leq \sup_{\R^n}\left|u_0^h\right|+\frac{\Gam(2-\al)M_0}{(1-\al)\al}t^\alpha
\quad\text{for all} \ 
h>0, x\in \R^n, t\geq 0.
\]
\end{lem}
\begin{proof}
We define
\[
W^h(x, t):=\sup_{\R^n}\left|u_0^h\right|+\frac{\Gam(2-\al)M_0}{(1-\al)\al}V^h(x, mh)
\]
for any $(x, t)\in \R^n\times [0, \infty)$, where $m=\lfloor t/h\rfloor$ and $V^h$ is given in Lemma \ref{lem:est1}.
In light of Lemma \ref{lem:est1}, we have 
\[
\pl_t^{\al, h} W^h(x,mh)+F\left(x, t, DW^h(x,mh), D^2W^h(x,mh)\right)
\ge M_0+F(x, t, 0, 0) \ge 0
\]
for all $m\in\N$. 
Combining with $U^h(\cdot,0)\le W^h(\cdot,0)$ on $\R^n$, by Proposition \ref{prop:monotone}, we get $U^h(\cdot,mh)\le W^h(\cdot,mh)$ for all $m\in\N$. 
Symmetrically, we get $U^h(x, mh)\ge -W^h(\cdot,mh)$ for all $m\in\N\cup\{0\}$, 
which implies the conclusion.  
\end{proof}

\section{Convergence of discrete schemes}\label{sec:main}

Let $u^h$ be the function defined by \eqref{def:Uh}. By Lemma \ref{lem:bound} and \eqref{initial approx}, we can define the half-relaxed limit of $u^h$ as follows:
\begin{equation}\label{half-relax}
\begin{aligned}
\ol{u}(x,t)&:=
\lim_{\delta\to 0}\sup\left\{u^h(y,s): |x-y|+|t-s|\le \delta,\ s\geq 0,\ 0<h\le \delta\right\},\\
\ul{u}(x,t)&:=
\lim_{\delta \to 0}\inf\left\{u^h(y,s): |x-y|+|t-s|\le \delta,\ s\geq 0,\ 0<h\le \delta\right\} 
\end{aligned}
\end{equation}
for all $(x,t)\in\R^n\times[0,\infty)$.

By the definition of Riemann integral and the operator $\pl_{t}^{\al,h}$, we have the following.
\begin{lem}\label{lem:limit}
Let $\pl_t^{\alpha, h}$ be given by \eqref{eq:op discrete}. Then for any $\psi\in C^{1}(\R^n\times[0,\infty))$, we have 
\[
\pl_{t}^{\al, h}\psi\to \pl_t^{\al}\psi 
\quad
\text{locally uniformly in} \ \R^n\times(0,\infty) \  \text{as}\ h\to0. 
\] 
\end{lem}

\begin{prop}[Sub- and supersolution property]\label{prop:half}
Let $\ol{u}$ and $\ul{u}$ be the functions defined by \eqref{half-relax}. 
Then $\ol{u}$ and $\ul{u}$ are, respectively, a subsolution and supersolution to \eqref{eq:1}.  
\end{prop}

\begin{proof}
We only prove that $\ol{u}$ is a subsolution to \eqref{eq:1} as we can similarly prove that 
$\ul{u}$ is a supersolution to \eqref{eq:1}.  

Take a test function $\varphi\in C^2(\R^n\times[0,\infty))$ and 
$(\hat{x},\hat{t})\in \R^n\times(0,\infty)$ so that $\ol{u}-\varphi$ takes 
a strict maximum at $(\hat{x},\hat{t})$ with $(\ol{u}-\varphi)(\hat{x},\hat{t})=0$. 
By adding $|x-\hat{x}|^4$ to $\varphi$ (we still denote it by $\varphi$), 
we may assume that $\varphi(x,t)\to \infty$ as $|x|$ uniformly for all  $t\geq 0$. 

We first claim that there exists $(x_{j},t_j)\in\R^n\times(0,\infty)$, $h_j>0$ so that 
$(x_j,t_j)\to(\hat{x},\hat{t})$ and  $h_j\to0$ as $j\to\infty$, 
\begin{align}
&u^{h_j}(\cdot, t_j)-\varphi(\cdot,t_j) \ \text{takes a maximum at} \ x_{j},  
\label{claim2}\\ 
&
\sup_{(x,t)\in\R^n\times(0,\infty)}(u^{h_j}-\varphi)(x,t)
<(u^{h_j}-\varphi)(x_j,t_j)+h_j.   \label{claim3}
\end{align}
Indeed, by definition of $\ol{u}$, there exists $(y_j,s_j)\in\R^n\times(0,\infty)$, and $h_j>0$ so that 
\begin{align*}
&(y_j,s_j)\to(\hat{x},\hat{t}), \ h_j\to0, 
\ \text{and} \ 
u^{h_j}(y_j,s_j)\to\ol{u}(\hat{x},\hat{t}) \quad\text{as} \ j\to\infty.  
\end{align*}
We next take $t_j>0$ such that 
\[
\sup_{(x,t)\in\R^n\times(0,\infty)}(u^{h_j}-\varphi)(x,t)
<
\sup_{x\in\R^n}(u^{h_j}-\varphi)(x,t_j)+h_j. 
\]
Also, by Lemma \ref{lem:bound} again, there exists $x_j\in\R^n$ so that 
\[
\sup_{x\in\R^n}(u^{h_j}-\varphi)(x,t_j)
=\max_{x\in\R^n}(u^{h_j}-\varphi)(x,t_j)
=(u^{h_j}-\varphi)(x_j,t_j). 
\]
Then, we can also easily check that 
$(x_j,t_j)\to(\hat{x},\hat{t})$ as $j\to\infty$. 
\medskip

Set $N_j:=\lfloor t_j/h_j\rfloor$.  
Then we have $u^{h_j}(\cdot ,t_j)=U^{h_j}(\cdot, N_jh_{j})$ in $\R^n$. 
Since $U^{h_j}(\cdot, N_jh_{j})$ is a viscosity solution to \eqref{eq:m} with $m=N_j$ and $h=h_j$, 
in light of \eqref{claim2}, 
we obtain 
\[
\pl_{t}^{\al, h_j}u^{h_j}(x_{j}, t_j)+F\left(x_j, t_j, D\varphi(x_{j},t_{j}),D^2\varphi(x_{j},t_{j})\right)\le 0. 
\]
Set 
$\sig_j:=\max_{x\in\R^n}(u^{h_j}-\varphi)(x,t_j)
=u^{h_j}(x_j,t_j)-\varphi(x_j, t_j)$.  
In light of \eqref{claim3}, we have 
\[
(u^{h_j}-\varphi)(x_j,kh_j)
\le h_j+\sig_j
\]
for all $k=0,\ldots,N_j-1$. 
Hence,  
\begin{align*}
& \Gam(2-\al)(h_j)^\al\pl_{t}^{\al, h_j}u^{h_j}(x_j, t_j)=
u^{h_j}(x_j,N_jh_j)
-\sum_{k=0}^{N_j-1} C_{N_j,k}u^{h_j}(x_j, kh_j)\\
\ge&\,  
\varphi(x_j,N_jh_j)+\sig_j
-\sum_{k=0}^{N_j-1} C_{N_j, k}\left(\varphi(x_j, kh_j)+h_j+\sig_j\right).
\end{align*}
Noting that 
\begin{equation*}\label{eq:sum}
\sum_{k=0}^{N_j-1} C_{N_j, k}=f(1)=1, 
\end{equation*} 
we obtain 
\begin{align*}
\pl_{t}^{\al,h_j}u^{h_j}(x_j,N_jh_j)
\ge&\,  \frac{1}{\Gam(2-\al)(h_j)^\al}\left\{
\varphi(x_j,N_jh_j)-\sum_{k=0}^{N_j-1} C_{N_j, k}\varphi(x_j, kh_j)-h_j\right\}\\
=&\,  
\pl_{t}^{\al,h_j}\varphi(x_j,N_jh_j)+O(h_j^{1-\alpha}). 
\end{align*}
We therefore obtain 
\[
\pl_{t}^{\al,,h_j}\varphi(x_{j},N_jh_j)
+F\left(x_j, t_j, D\varphi(x_{j},t_{j}),D^2\varphi(x_{j},t_{j})\right)
\le O(h_j^{1-\alpha}).
\]
By Lemma \ref{lem:limit} and the continuity of $F$, sending $j\to\infty$ yields 
\[
\pl_{t}^{\al}\varphi(\hat{x},\hat{t})+
F\left(\hat{x}, \hat{t}, D\varphi(\hat{x},\hat{t}),D^2\varphi(\hat{x},\hat{t})\right)
\le 0. 
\qedhere
\]
\end{proof}

\begin{prop}[Initial consistency]\label{prop:ini}
Assume that \eqref{eq:op bound} and {\rm(H1)} hold.  Let $\ol{u}$ and $\ul{u}$ be the functions defined by \eqref{half-relax}.
Then $\ol{u}\leq u_0\leq \ul{u}$ in $\R^n$.
\end{prop}

\begin{proof}
Fix any $x_0\in \R^n$. Since $u_0\in BUC(\R^n)$ and \eqref{initial approx} holds, for any $\sigma>0$ we can find a bounded smooth function $\phi_\sigma$ such that 
$\phi_\sigma(x_0)\leq u_0(x_0)+\sigma$ and 
$u_0^h(x)\leq \phi_\sigma(x)$ for all $x\in \R^n$ and all $h>0$ small.
 We claim that 
\[
\phi^h(x , t)=\phi_\sigma(x)+{M_{R_\sigma}\Gamma(2-\alpha)\over (1-\alpha)\alpha}t^\alpha
\]
is a supersolution of \eqref{eq:discrete} with $h>0$ small, where $M_{R_\sigma}$ is given in \eqref{eq:op bound} with 
\[
R_\sigma=\sup_{\R^n}\left(|D \phi_\sigma|+|D^2 \phi_\sigma|\right).
\]
Indeed, for any $x\in \R^n$, applying Lemma \ref{lem:est1}, 
we deduce that for all $m\in\N$,
\[
\partial_t^{\alpha, h}\phi^h(x, mh)\geq M_{R_\sigma}\geq -F\left(x, mh, D \phi_\sigma(x), D^2 \phi_\sigma(x)\right)
\]
for all $x\in \R^n$.
We thus can adopt Proposition \ref{prop:monotone} to obtain that 
$u^h(x, Nh)\leq \phi^h(x, Nh)$ for all $x\in \R^n$ and $t\geq 0$ with $N=\lfloor t/h\rfloor$,
which implies that 
\[
u^h(x, t)\leq \phi_\sigma(x)+{M_{R_\sigma}\Gamma(2-\alpha)\over (1-\alpha)\alpha}t^\alpha
\]
for all $x\in \R^n$ and $t\geq 0$. We thus have 
\[
\overline{u}(x_0, 0)\leq \phi_\sigma(x_0), 
\]
which implies, by letting $\sigma\to 0$,  that 
$\overline{u}(x_0, 0)\leq u_0(x_0)$. 
The proof for the part on $\underline{u}$ is symmetric and therefore omitted here. 
\end{proof}

\begin{proof}[Proof of Theorem {\rm\ref{thm:main}}]
If (H2) holds, then the conclusion of the theorem is a straightforward result of Propositions \ref{prop:half} and \ref{prop:ini}. 
\end{proof}

\end{document}